\documentclass[12pt]{extarticle}
\usepackage[a4paper,bindingoffset=0.in,left=2.5cm,right=2.5cm,top=2cm,bottom=3cm]{geometry}
\usepackage[numbers]{natbib}  

\usepackage{graphicx} 
\usepackage{mathtools,amsmath,amssymb,amsthm}
\usepackage{mathrsfs}
\usepackage{comment,xcolor}
\usepackage{multirow}
\usepackage{makecell}
\usepackage{array,booktabs}
\usepackage{color, colortbl}
\usepackage{tabstackengine}[2016--11--30]
\usepackage{xcolor}
\usepackage[labelfont=bf]{caption}
\usepackage{algpseudocode}
\usepackage{siunitx}
\usepackage[affil-it]{authblk}
\usepackage[ruled,vlined,linesnumbered]{algorithm2e}
\usepackage{tikz}
\usepackage{enumitem}
\usetikzlibrary{decorations.pathreplacing,decorations.markings,arrows}

\tikzset{
  midarrow/.style={
    postaction={
      decorate,
      decoration={
        markings,
        mark=at position 0.45 with {\arrow{>}}
      }
    }
  }
}
\newtheorem{theorem}{Theorem}[section]

\newtheorem{corollary}[theorem]{Corollary}

\newtheorem{problem}[theorem]{Problem}
\theoremstyle{definition} 
\newtheorem{definition}[theorem]{Definition}
\newtheorem{remark}[theorem]{Remark}
\newtheorem{example}[theorem]{Example}
\renewenvironment{proof}{{\noindent\bfseries Proof.}}{\qed} 

\usepackage[colorlinks=true,allcolors=blue!75!black,backref=page]{hyperref}
\definecolor{ThistleBlue}{rgb}{102, 153, 255}

\newcommand{\diag}{\operatorname{diag}}

\newcommand{\adj}{\operatorname{adj}}
\newcommand{\sign}{\operatorname{sign}}
\newcommand{\Gr}{\operatorname{Gr}}

\newcommand{\Id}{I}
\newcommand{\rowspan}{\operatorname{rowspan}}
\newcommand{\R}{\mathbb{R}}
\newcommand{\Ps}{\mathbb{P}}

\newcommand{\ba}{\mathbf{a}}
\newcommand{\bB}{\mathbf{B}}

\newcommand{\xstar}{\mathbf{x}_{\star}}
\newcommand{\Pl}{\operatorname{Pl}}

\newcommand{\baseset}{\binom{[2n]}{n}}

\newcommand{\GrR}{\Gr_{\mathbb{R}}}
\newcommand{\Rn}{\mathbb{R}^n}
\newcommand{\Rnn}{\mathbb{R}^{n\times n}}  
\newcommand{\Rtwon}{\mathbb{R}^{2n}}
\newcommand{\Rntwon}{\mathbb{R}^{n\times 2n}}  

\newcommand{\feascond}{\adj(\bB)\ba}          
\newcommand{\detB}{\det(\bB)}                 
\newcommand{\JacAB}{J_{\ba,\bB}}              
\newcommand{\Grnn}{\Gr(n,2n)}                 
\newcommand{\GrRnn}{\GrR(n,2n)}               

\renewcommand{\L}{\mathcal{L}}

\DeclareMathOperator{\Tr}{tr}

\DeclareMathAlphabet{\mathcal}{OMS}{cmsy}{m}{n}
\DeclareMathAlphabet{\mathcalbf}{OMS}{cmsy}{b}{n}
\DeclareMathAlphabet{\mathbfsfit}{\encodingdefault}{\sfdefault}{b}{it}
\DeclareMathAlphabet{\mathbfsf}{\encodingdefault}{\sfdefault}{b}{n}
\renewcommand{\vec}[1]{\boldsymbol{#1}}

\newenvironment{coordeq}
{\begin{equation}\begin{array}{@{}ll@{}}}
{\end{array}\end{equation}}

\newcommand{\coord}[1]{#1 &}

\renewcommand{\leq}{\leqslant}
\renewcommand{\geq}{\geqslant}

\title{{\bf Strata of Ecological Coexistence via Grassmannians}}
\author{\normalsize{Türkü Özlüm Çelik, Pierre A. Haas, Georgy Scholten, Kexin Wang, Giulio Zucal}}
\date{\small\today}
\begin{document}

\maketitle

\begin{abstract}
    We study the Lotka--Volterra system from the perspective of computational algebraic geometry, focusing on equilibria that are both feasible and stable. These conditions stratifies the parameter space in $\Rn \times \Rnn$ with the feasible-stable semialgebraic sets. We encode them on the real Grassmannian $\operatorname{Gr}_\mathbb{R}(n,2n)$ via a parameter matrix representation, and use oriented matroid theory to develop an algorithm, combining Grassmann--Plücker relations with branching under feasibility and stability constraints. This symbolic approach determines whether a given sign pattern in the parameter space $\Rn \times \Rnn$ admits a consistent extension to Plücker coordinates. As an application, we establish the impossibility of certain interaction networks, showing that the corresponding patterns admit no such extension satisfying feasibility and stability conditions, through an effective implementation. We complement these results using numerical nonlinear algebra with \texttt{HypersurfaceRegions.jl} to decompose the parameter space and detect rare feasible-stable sign patterns. 
\end{abstract}


\section{Introduction}
Let $n$ be a positive integer, $\ba \in \Rn$ a real vector, and $\bB \in \Rnn$ a real matrix. The generalized {Lotka--Volterra} system~\cite{murray,Hofbauer1998,Pierre} for the time-varying vector $\mathbf{x}\in\Rn$ is defined by the system of ordinary differential equations (ODEs):
\begin{equation}\label{eq:LotkaVolterra}
    \dot{\mathbf{x}} = \diag(\mathbf{x})(\ba - \bB\mathbf{x}),
\end{equation}
where $\operatorname{diag}(\mathbf{x})$ denotes the diagonal matrix with the components of $\mathbf{x}$ on the diagonal. In components, we write $\ba=(a_i)$, $\bB=(b_{ij})$, $\mathbf{x}=(x_i)$, so that \eqref{eq:LotkaVolterra} becomes 
\begin{equation}\label{eq:LotkaVolterraComponent}
    \dot{x}_i = x_i \left(a_i - \sum_{j=1}^n b_{ij} x_j\right), \quad \text{for } i = 1, \dots, n.
\end{equation}
At a \emph{steady state} of the system, $ \dot{\mathbf{x}}={\bf 0}$. If $\bB$ is nonsingular, the Lotka--Volterra system has a steady state,
\begin{equation}\label{eq:steadyLotka}
\xstar = \bB^{-1} \ba.
\end{equation}
Of particular interest to us will be steady states with all coordinates positive, i.e., steady states of \emph{coexistence}. Borrowing ecological parlance, we term these states \emph{feasible}~\cite{roberts74,Pierre}. The steady state in \eqref{eq:steadyLotka} is the only possible steady state of coexistence in the Lotka--Volterra system \cite{Pierre}. Beyond feasibility, we study \emph{stability}: a feasible equilibrium is (locally asymptotically) stable if small perturbations decay over time, which occurs exactly when all eigenvalues of the Jacobian at $\xstar$ have negative real parts \cite{allen2007introduction}.
Explicitly, this Jacobian is
\begin{equation}\label{eq:Jac}
J_{\ba,\bB}(\xstar) = -\mathrm{diag}(\xstar)\bB.
\end{equation}
A necessary condition for the equilibrium to be stable is $(-1)^n\mathrm{det}({J}_{\ba,\bB})>0$. This is a straightforward consequence of the Routh--Hurwitz conditions for stability~\cite{murray}. Now feasibility requires $\mathrm{det}(\mathrm{diag}(\xstar))>0$, so a necessary condition for stability and feasibility is ${\det(\bB)>0}$. Writing $\bB^{-1}$ in \eqref{eq:steadyLotka} in terms of the adjugate matrix $\mathrm{adj}(\bB)$, we can therefore express the conditions of feasibility and stability equivalently using the normalized forms of $\xstar$ and $J_{\ba,\bB}$,
\begin{align}\label{eq:steadyLotkaadjugate}
\tilde{\mathbf{x}}_\star := \mathrm{adj}(\bB)\ba,&&\tilde{J}_{\ba,\bB}(\tilde{\mathbf{x}}_\star)=-\mathrm{diag}(\tilde{\mathbf{x}}_\star)\bB.
\end{align}

In theoretical ecology, the Lotka--Volterra system is the simplest model for the dynamics of a vector $\mathbf{x}$ of abundances of $n$ species~\cite{murray, Hofbauer1998}. In this context, $\ba$ represents the growth rates of these species in the absence of interactions, which are encoded by $\bB$. In more detail, $a_i>0$ for a species that grows on its own, but $a_i<0$ for a species that dies in the absence of other species. The diagonal elements $b_{ii}$ describe intraspecies interactions, and it is usual to assume $b_{ii}>0$ to impose a notion of carrying capacity. Different types of interspecies interactions are defined by $b_{ij}$ ($i\neq j$): the interaction between species $i$ and $j$ is \emph{competitive} if $b_{ij},b_{ji}>0$; it is \emph{mutualistic} if $b_{ij},b_{ji}<0$; and it is a \emph{directed predator-prey interaction}, with species $i$ predating on species $j$, if $b_{ij}<0$, $b_{ji}>0$.

Coexistence of species in the Lotka--Volterra system corresponds to the existence of a stable and feasible steady state $\xstar$. It is straightforward to decide, numerically, whether $\xstar$ is stable and feasible for given $\ba$ and $\bB$, but theoretical ecology ultimately seeks to describe the general principles that determine the possibility of stable and feasible coexistence. In particular, one asks: How does the possibility of stable and feasible coexistence depend on the network of types of ecological interactions, i.e., on the sign patterns of $\ba$ and $\bB$?

The explicit conditions for feasibility and stability are too complex, however, to allow immediate insights into this question beyond the textbook case $n=2$~\cite{murray}. Much recent work in theoretical ecology, exemplified by~\cite{allesina12,mougi12,coyte15,butler18,hatton24}, has therefore followed the approach pioneered by Robert May~\cite{may72}, of studying large ecological communities statistically, using results from random matrix theory~\cite{allesina15}. This approach has revealed, for example, the generic effects of competition, mutualism, and predation on stability and feasibility of coexistence, but it cannot capture the full network structure of ecological interactions. Very recent work~\cite{Pierre} has therefore taken a complementary approach, of exhaustive analysis of all networks of types of ecological interactions for small communities with $2\leq n\leq 5$. This revealed that the probability of stable and feasible coexistence depends hugely on the structure of this network. In particular, this work discovered that a very small subset of these networks are \emph{impossible ecologies} in which stable and feasible coexistence is nontrivially impossible~\cite{Pierre}, but the identification of impossible ecologies with $n>3$ remained a numerical conjecture.


Here, we therefore bring the perspective of computational and real algebraic geometry~\cite{michalek2021invitation} to the study of the feasible~\eqref{eq:steadyLotka} and stable~\eqref{eq:Jac} equilibria in the parameter space $(\ba,\bB) \in \Rn \times \Rnn$ of the Lotka--Volterra system~\ref{eq:LotkaVolterra}. We show how the requirements of feasibility and stability translate into a collection of polynomial inequalities in the entries of $\ba$ and $\bB$, so that the regions of parameter space corresponding to feasible-stable equilibria are \emph{semialgebraic sets} in $\mathbb{R}^{n+n^2}$, the \emph{feasible-stable strata of ecological coexistence}. The boundaries of these sets encode natural degeneracies such as equilibria with vanishing coordinates or marginally stable dynamics with eigenvalues on the imaginary axis. Our primary focus is on understanding the structure of these semialgebraic sets, in particular to identify those sign patterns for $(\ba,\bB)$ that admit no feasible-stable equilibrium, which are precisely the impossible ecologies of~\cite{Pierre}.

To address this question, we recast the problem in an algebraic-combinatorial framework on the real Grassmannian $\mathrm{Gr}_{\mathbb{R}}(n,2n)$. Through the matrix representation\linebreak ${[\mathrm{diag}(\ba) \mid \bB]}$, feasibility and stability conditions become sign constraints on the Plücker coordinates, which can be analyzed via oriented matroid theory \cite{oriented_mat}. 
This reformulation allows us to replace parts of the semialgebraic feasibility-stability problem by purely combinatorial constraints derived from the Grassmann--Plücker relations, providing a systematic route to proving impossibility results. 
The Grassmannian perspective also lends itself to effective computation: we develop a search algorithm, combining propagation via Grassmann--Plücker relations with branching steps constrained by ecological feasibility and stability checks. This approach not only certifies impossibility for certain sign patterns, but also organizes the feasible cases into strata indexed by their oriented matroid type, giving a combinatorial stratification of the parameter space that mirrors its semialgebraic decomposition.

\begin{example}[$n=2$]\label{ex:n2}
As a motivating example, we consider the two-species Lotka--Volterra system. Feasibility corresponds to $\xstar > 0$, or equivalently $\feascond > 0$, and stability requires both $\Tr(-\mathrm{diag}(\tilde{\mathbf{x}}_\star)\bB) >0$ and $\detB > 0$. For instance, with
\[
\bB = \begin{bmatrix} b_{11} & b_{12} \\ b_{21} & b_{22} \end{bmatrix}, \quad 
\ba = \begin{bmatrix} a_1 \\ a_2 \end{bmatrix}, \quad 
\feascond = 
\begin{bmatrix}
b_{22}a_1 - b_{12}a_2 \\
- b_{21}a_1 + b_{11}a_2
\end{bmatrix},
\]
feasibility and stability translate into explicit polynomial inequalities. Embedding the system into the Grassmannian $\Gr(2,4)$ via the $2\times 4$ matrix representation $M=[\diag(\ba)|\bB]$,
we recover the entries of $\diag(\ba)\bB$ and $\detB$ as Plücker coordinates, i.e., two-by-two minors of $M$: 
\begin{equation*}
p_{13} = a_1 b_{21}, \quad p_{14} = a_1 b_{22}, \quad p_{23} = -a_2b_{11}, \quad p_{24} = -a_2b_{12}, \quad p_{34} = \detB,
\end{equation*}
and feasibility becomes the inequalities 
$p_{14} + p_{24} > 0$ and  $-p_{13} - p_{23} > 0$. Each possible combination of signs in the off-diagonal entries of $\bB$ and $\ba$ corresponds to a total of 10 different networks~\cite{Pierre} of types of ecological interactions (i.e., competition, mutualism, predation), four of which are incompatible with feasibility. 
For instance, under \emph{obligate mutualism} ($a_i < 0$, $b_{ij} < 0$ for $i \ne j$), $p_{14}<0,p_{24}<0$, so the feasibility inequalities cannot hold, making this network impossible.
\end{example}

In Section~\ref{sec:RouthHurwitz}, we show that the feasibility and stability conditions of the Lotka–Volterra system define a semialgebraic set, which we refer to as the \emph{feasible–stable stratum} (Definition~\ref{def:FSstratum}), defined by polynomial inequalities in the input parameters $(\ba,\bB) \in \Rn\times \Rnn$. Section~\ref{sec:geometry} introduces the \emph{ecological Grassmannian}, where we propose the real Grassmannian $\GrRnn$ as an algebraic framework for studying these semialgebraic constraints. 
In Proposition~\ref{prop:toalgebraicGrassmannian}, we reformulate the feasible-stable stratification over the Grassmannian, and in Corollary~\ref{thm:comb_grass_gs} we present a purely combinatorial relaxation of this reformulation in terms of sign patterns on the input parameters. In Section~\ref{sec:computation}, we develop algorithms that exploit the real Grassmannian and its oriented matroid stratification to explore the feasible-stable semialgebraic sets efficiently, and we describe their implementation in \textsc{Maple}~\cite{gitlab} based on our findings in Section~\ref{sec:geometry}. Section~\ref{Sec:Experim} reports the results of our computational experiments. We answer questions left open in~\cite{Pierre} in Theorem~\ref{thm:impossible4}, which establishes the impossibility of certain ecological networks with $n=4$ species, as illustrated in Figure~\ref{fig:n4impossible}. Finally, Section~\ref{Sec:HypersurfacesComp} uses tools from numerical algebraic geometry to study the feasible-stable parameter space via hypersurface arrangements, yielding explicit region counts, representative feasible–stable points, and combinatorial statistics on sign patterns.

\section{Feasible Stable Ecologies as Semialgebraic Sets:\\Routh--Hurwitz Polynomials}
\label{sec:RouthHurwitz}


The condition that all eigenvalues of an $n \times n$ matrix $M$ have negative real parts can be reformulated as positivity conditions on the coefficients of its characteristic polynomial. This formulation, known as \emph{Routh--Hurwitz criterion}, goes back to the classical works of Routh (1877) and Hurwitz (1895). For modern treatments and applications, 
we refer the reader to~\cite{allen2007introduction,FULLER196871,murray}. We adopt this criterion as our point of departure. Let the characteristic polynomial of $M$ be given by
\begin{align}\label{eq:charpoly}
  p_M(\lambda) &= \det( \lambda \Id - M) = \lambda^n + c_{n-1}\lambda^{n-1} + \ldots +c_1 \lambda+  c_0.
\end{align}
The Hurwitz polynomial of $M$ of degree $k$ (for $k = 1, \dots, n$) is defined as the determinant of the following $k\times k$ matrix, where $c_k=0$ if $k<0$:
    \begin{equation}\label{eq:Hurwitzdeterminant}
\begin{pmatrix}
c_{n-1} & c_{n-3} & c_{n-5} & \cdots & c_{n-(2k-1)} \\
1     & c_{n-2} & c_{n-4} & \cdots & c_{n-(2k-2)} \\
0       & c_{n-1} & c_{n-3} & \cdots & \vdots \\
0       & 1     & c_{n-2} & \cdots & \vdots \\
\vdots  & \vdots  & \vdots  & \ddots & \vdots \\
0       & 0       & 0       & \cdots & c_{n-k}
\end{pmatrix}.
\end{equation}

\begin{theorem}[Routh--Hurwitz criterion]\label{thm:Routh}
The polynomial in~\eqref{eq:charpoly} has all roots with negative real part if and only if all the Hurwitz polynomials~\eqref{eq:Hurwitzdeterminant} are positive.
\end{theorem} 

Alternative formulations of this criterion can be found in~\cite{FULLER196871}.
Indeed, the stability condition can be expressed in several equivalent ways: for example, in addition to the positivity of the Hurwitz determinants, one may require that all coefficients $c_i$ of the characteristic polynomial be 
positive. In particular, a necessary condition for stability is $0<c_0=p_M(0)=(-1)^n\det(M)$. In our setting, we may choose among these formulations to suit the algebraic manipulations and computational tools applied in later sections.

From~\eqref{eq:charpoly}, it follows that the coefficients $c_i$ are polynomials in terms of the entries of the matrix $M$.
For a given $(\ba,\bB) \in \Rn\times \Rnn$, we
take $M$ to be the Jacobian matrix $\JacAB(\textbf{x}_{\star})$ in~\eqref{eq:Jac} evaluated at the steady state $\textbf{x}_{\star}=\bB^{-1} \ba$. 
Although these constraints are rational in $(\ba,\bB)$, we can use the interdependence between stability and feasibility to pass to a semialgebraic formulation of the constraints, as noted above around Equation~\eqref{eq:Jac}: with $\det(\bB)>0$, feasibility $\xstar>0$ is equivalent to $\tilde{\mathbf{x}}_\star>0$, and $\JacAB$ is Hurwitz-stable if and only if $\tilde{J}_{\ba,\bB}$ is.
Thus our goal is to study the parameter space $\Rn\times \Rnn$ for ecologies via polynomial positivity conditions in $a_i$, $b_{ij}$ and $c_i(\ba,\bB)$ using tools from computational algebraic geometry. 
With this in place, we define the following semialgebraic set as the geometric object corresponding to the feasible-stable ecologies $(\ba,\bB) \in \Rn\times \Rnn$. 


\begin{definition}[Feasible-stable stratum]\label{def:FSstratum}
Fix an ordered sign pattern
\[
   \sigma=(\sigma_1,\dots, \sigma_n,\sigma_{11}, \sigma_{12}, \dots , \sigma_{nn}) 
      \;\in\; \{\pm\}^{\,n + n^2} .
\]
The \emph{feasible-stable stratum} associated to~$\sigma$ is the semialgebraic set
\begin{equation}\label{eq:feas_stable_stratum}
   \mathcal{S}_n(\sigma) \;=\;
   \left\{\,(\mathbf a,\bB)\in\Rn\!\times\Rnn \;\Bigg|\;
   \begin{aligned}
      &\sign(\ba,\bB)=\sigma, \quad \adj(\bB)\,\ba > \mathbf 0,\\[2pt]
      &\text{all the Hurwitz polynomials of } \tilde{J}_{\ba,\bB} > 0  
   \end{aligned}
   \right\}, 
\end{equation}
where $\tilde{J}_{\ba,\bB}$ is defined in~\eqref{eq:Jac}. 
Throughout, the polynomial constraints in $(\ba,\bB)$ ensuring $\adj(\bB)\ba>0$ are called the \emph{feasibility conditions}.
The polynomial constraints induced by the Routh--Hurwitz criterion are called the \emph{stability (Hurwitz) conditions}.
\end{definition}

In Definition~\ref{def:FSstratum}, we impose strict inequalities in the feasibility and stability conditions, so that each $\mathcal{S}_n(\sigma)$ is an open basic semialgebraic subset of $\Rn \times \Rnn$.
From the lens of real algebraic geometry, it would be more natural to define a closed analogue by replacing the strict inequalities with non-strict ones, thereby including boundary points where one or more feasibility or stability polynomials vanish.
Such boundary cases correspond to degenerate steady states (in which some coordinates of $\xstar$ are equal to zero) or to marginally stable dynamics (for which the Jacobian has eigenvalues on the imaginary axis).
Since our aim is to characterize strictly positive equilibria that are asymptotically stable, we restrict attention to the open formulation and exclude these limiting configurations.


\begin{definition}[Feasible-stable stratification]\label{def:stratification}
The collection of strata $\bigl\{\mathcal{S}_{n}(\sigma)\bigr\}_{\sigma}$
provides a pairwise disjoint, locally closed partition of the subset of the parameter space
$\Rn\times\Rnn$ in which the feasibility and stability conditions are satisfied.
We call this decomposition the \emph{feasible-stable stratification} of
the Lotka--Volterra parameter space. Each stratum is labeled by its
sign pattern~$\sigma$ and records precisely those parameters that realize
a feasible and asymptotically stable equilibrium contained in the orthant of $\Rn\times \Rnn$ designated by $\sigma$. The \emph{coexistence locus} in parameter space is the union of all
\emph{nonempty} feasible-stable strata. A Lotka--Volterra network with sign pattern~$\sigma$ is called
\emph{impossible} if and only if its feasible-stable stratum is empty,
\(\mathcal{S}_{N}(\sigma)=\varnothing\).
Thus, no choice of parameters with the prescribed signs yields a strictly
positive, asymptotically stable equilibrium.
\end{definition}

\begin{problem}[Realizability]\label{prob:feasible_ecology}
  Given a sign pattern $\sigma \in \{\pm\}^{\,n + n^2}$, decide whether the feasible-stable stratum $\mathcal{S}_{n}(\sigma)$ is empty. This is precisely the impossible-ecology problem studied in~\cite{Pierre}.
\end{problem}

To illustrate the semialgebraic formulation, we examine the following example.

\begin{example}[$n = 3$ and the impossibility of obligate mutualism]\label{ex:n_3}
We consider \emph{obligate mutualism} \cite{Pierre} of $n=3$ species, corresponding to the sign pattern $\sigma$ with $\sigma_i=-$ for all $i$, $\sigma_{ii}=+$, and $\sigma_{ij}=-$ for $i\neq j$. All off-diagonal interactions are mutualistic and all intrinsic growth rates are negative. The feasibility inequality $\feascond>0$ under this $\sigma$ forces all $2\times 2$ principal minors of $\bB$ to be negative. This contradicts the Hurwitz stability condition $c_1>0$, which follows from the expression of $c_1$ in terms of all principal minors. Hence $\mathcal{S}_3(\sigma)=\varnothing$, so this network is an impossible ecology.
\end{example}

We conclude the section by relating our framework to a broader context. Population dynamics have been studied not only in ecology but also in the context of game theory~\cite{Hofbauer1998,Zeeman}, where similar mathematical structures arise. 
In particular, the equations governing steady states in the Lotka–Volterra system share the same algebraic form as equilibrium conditions in game theory 
\cite[Equation (2.3)]{AboEtAl2025VectorBundle}.
In both cases, the defining relations are polynomial constraints involving an interplay matrix, viz., the payoff matrix in game theory and the interaction matrix in ecology.
From this viewpoint, the ecological steady-state condition can be interpreted as an analogue of the best-response equilibrium condition, suggesting that computational algebraic geometry methods developed for game-theoretic equilibria may be applied to feasibility problems in ecological models. 
A similar connection exists with other disciplines: the Lotka--Volterra system has actually been studied in the contexts of chemical reaction networks~\cite{Dickenstein} and integrable systems~\cite{BOGOYAVLENSKY198834, WeiGengZeng2019Bogoyavlensky}, both of which employ algebraic and geometric techniques central to nonlinear algebra. Formulating these connections precisely, within the perspective of nonlinear algebra, is a promising avenue for future research.

\section{Feasible Stable Ecologies via Grassmannians}
\label{sec:geometry}

In this section, we map the parameters $(\ba,\bB)\in \Rn \times \Rnn$ into the Grassmannian $\GrRnn$, the set of all $n$-dimensional linear subspaces of the vector space $\mathbb{R}^{2n}$. We study the feasible-stable strata $\mathcal{S}_n(\sigma)$ of the Lotka--Volterra system~\eqref{eq:LotkaVolterra} via the algebra and geometry of the Grassmannian. Throughout this section, we omit the subscript $\mathbb{R}$, assuming that we are working only over $\mathbb{R}$. We refer the reader to \cite[Chapter 5]{michalek2021invitation} for a perspective on Grassmannians and projective spaces from nonlinear algebra. 
We first outline the construction of the map from the parameters $(\ba,\bB)$ to the real Grassmannian and defer the algebraic details to later in the section. 
The following construction defines the pipeline of our work:
\begin{equation}
\label{eq:parameter_pipeline}
(\ba,\bB) \longmapsto \L = \rowspan(M) 
\longmapsto p_{\ba,\bB} = \Pl(\L),
\end{equation}
where $\Pl$ denotes the \emph{Pl\"ucker embedding} of the Grassmannian into projective space and the parameters $(\ba,\bB)$ are encoded into the block matrix
\begin{equation}
   \label{eq:param_matrix}
   M = 
   \begin{bmatrix}
       a_1  & \cdots & 0   & b_{11}  & \cdots & b_{1n} \\
       \vdots &  \ddots & \vdots  & \vdots & \ddots & \vdots \\
       0 & \cdots & a_n & b_{n1}  & \cdots & b_{nn}
   \end{bmatrix}
   \in \Rntwon.
\end{equation}
We will denote by $M_S$ the square matrix obtained by selecting the columns indexed by the ordered set $S$. (Hereinafter, we fix the ordering of the elements in set notation because of the alternating sign properties of the determinant.) We start by observing that positivity of the steady-state vector $\xstar = \bB^{-1}\ba$ can be expressed in terms of the minors of $M$, 
\begin{equation}
\label{eq:xstar_global}
(\xstar)_i = (-1)^{i+1} \sum_{j=1}^n \det(M_{j \cup S_i})/\det(M_{[n+1,\dots,2n]}),
\end{equation}
where $S_i=[n+1,\ldots,2n]\setminus\{n+i\}$. From this observation, it is natural to attempt to reformulate Problem~\ref{prob:feasible_ecology} in terms of the $n \times n$ minors of $M$. 

Since linear row operations on $M$ scale all $n\times n$ minors by a common nonzero factor, we work projectively: we identify vectors of minors up to $v\sim\lambda v$ ($\lambda\neq0$). The $n$-dimensional subspace $\L = \rowspan(M) \subset \Rtwon$ can be uniquely determined by the vector of minors. The Pl\"ucker embedding is the map
$\Pl\colon\Grnn\hookrightarrow\Ps^{\binom{2n}{n}-1}$ into the real \emph{projective space} of dimension ${\binom{2n}{n}-1}$ which takes $\L \in \Gr(n,2n)$ to its \emph{Pl\"ucker coordinates} $p\coloneq(p_S)=\left(\det(M_S)\mid{S\in\baseset}\right)$, wherein $\baseset$ is the set of subsets of $[1,2,\dots,2n]$ of size~$n$. The image $\Pl(\Grnn)$ of the Pl\"ucker embedding is the zero locus of the classical \emph{Grassmann–Plücker relations}, which are homogeneous quadratic relations on the Plücker coordinates.  

Using~\eqref{eq:xstar_global}, the \emph{feasibility} of $\xstar$ can be expressed in Pl\"ucker coordinates as
\begin{equation}
\label{eq:feasibility_global}
(\xstar)_i=(-1)^{i+1} \sum_{j=1}^n \frac{p_{j \cup S_i}}{{p_{[n+1,\dots,2n]}}} > 0
\quad \text{for all } i = 1,\ldots,n. 
\end{equation}
The \emph{stability} inequalities from Theorem~\ref{thm:Routh} can translate into Pl\"ucker coordinates but require additional information from the vector $\ba$. 

We recall from~\eqref{eq:Jac} that $\JacAB(\xstar)$ is the Jacobian at the steady state $\xstar$ of system~\eqref{eq:LotkaVolterra}. 
Its characteristic polynomial is
\begin{align}\label{eq:charpolyJacobian}
    P_{\JacAB(\xstar)}(\lambda) &= \det\bigl(\lambda \Id - \JacAB(\xstar)\bigr) = \sum_{i=0}^{n} c_i\lambda^i,
\end{align}
with coefficients 
\begin{align}\label{eq:charpolyCoeff}
    c_i = \sum_{\substack{I \subseteq [n]\\ |I|=i}} \prod_{j\in I^\text{c}} (\xstar)_j \det(\bB_{I^{\text{c}}}),
\end{align}
where $I^{\text{c}} = [n]\setminus I$ and $\bB_{I^{\text{c}}}$ is obtained from $\bB$ by deleting the rows and columns indexed by $I$. 
For a subset $I \subseteq [n]$, we denote $\pi(I)$ as the inversion number of the set $[I^c, I]$, i.e., the minimal number of pairwise swaps required to sort it in ascending order. Then 
\begin{equation}\label{eq:pluckerMinorRelation}
p_{[I^{\text{c}},n+I]} = (-1)^{\pi(I)} \prod_{j\in I^{\text{c}}} a_j \det(\bB_{I^{\text{c}}}).
\end{equation}
Combining \eqref{eq:pluckerMinorRelation} and \eqref{eq:charpolyCoeff}, we obtain the expression 
\begin{equation}\label{eq:stability_plucker_global}
    c_{i} = \sum_{\substack{I \subseteq [n]\\ |I|=i}} (-1)^{\pi(I)} \prod_{j\in I} a_j \prod_{k \in {I^{\text{c}}}}(\xstar)_k \cdot \frac{p_{[I^{\text{c}},n+I]}}{p_{[1,\dots, n]}}.
\end{equation}
Thus from \eqref{eq:stability_plucker_global} and \eqref{eq:feasibility_global}, the stability conditions in Definition~\ref{def:FSstratum} can be expressed in terms of the Pl\"ucker coordinates and $\ba$.
The condition $c_0>0$ is equivalent to feasibility and $p_{[n+1,\dots,2n]}= \det(\bB)>0$ and it is the only among the Routh--Hurwitz conditions that is independent of $\ba$.
In projective space, a point $(p_S)\in\Gr(n,2n)$ is defined only up to a nonzero scalar; multiplying any row of $M$ by a negative factor flips the sign of every Pl\"ucker coordinate $p_S$. Feasibility fixes the sign of $p_{[1,\dots,n]}=a_1\cdots a_n$, which makes the sign of the Plücker coordinates well-defined. The positivity of $p_{[n+1,\dots,2n]}=\det(\bB)$ is then a necessary condition for the steady state to be feasible and stable. 


Our interest is now to study the sign patterns of the Pl\"ucker coordinates from the sign patterns of $(\ba,\bB)$ and the feasibility and stability conditions. We first state the semialgebraic stratum $\mathcal{S}_n(\sigma)$ in algebraic terms in the real Grassmannian ${\rm Gr}(n,2n)$: 

\begin{theorem}\label{prop:toalgebraicGrassmannian}
Let $\sigma\in\{-,+\}^{\,n}\times\{-,+\}^{\,n\times n}$ be an ordered sign assignment for $(\ba,\bB)$ and write $\sigma=(\sigma_a,\sigma_B)$ with $\sigma_a=(\sigma_1,\ldots,\sigma_n)$ and $\sigma_B=(\sigma_{ij})$. 
Then there exists a point $(\ba, \bB)$ in $\mathcal{S}_n(\sigma)$ if and only if there exists $p\in \Gr(n,2n)\subset \Ps^{\binom{2n}{n}-1}$ 
with $\sign(p_{[1,\dots,n]})=\sigma_1\cdots \sigma_n$ such that 
\begin{enumerate}[leftmargin=*]
    \item \label{it:1} the coordinates $p_{T_i\cup j}$ follow the  signs $\chi_{T_i\cup j}=\sigma_{ij}\prod_{k\in [n]\setminus i}\sigma_k$, where $T_i= [n]\setminus\{i\}$ for $i\in [n]$ and $j\in [n+1,\ldots, 2n]$, 
    \item\label{it:2} the inequalities~\eqref{eq:feasibility_global} hold,
    \item\label{it:3} the coefficients $c_i$ defined in~\eqref{eq:stability_plucker_global} for $i = 0,1,\ldots,n-1$ and all the corresponding Hurwitz determinants are positive.
\end{enumerate}
\end{theorem}
\begin{proof}
Let $(\ba, \bB)$ be a point in $\mathcal{S}_n(\sigma)$. 
The Pl\"ucker coordinates of the matrix $[\diag(\ba) \mid \bB]$ satisfy the inequalities~\eqref{eq:feasibility_global} and~\eqref{eq:stability_plucker_global} by construction.  

Conversely, let $p$ be the image of a linear space $\L$ in $\Gr(n, 2n)$ with non-zero coordinates $p_{[1,\ldots, n]}$ and $p_{T_i\cup j}$ for all $n^2$ choices of $T_i$ and $j$. 
Since $p_{[1,\ldots,n]}\neq0$, there exists $N\in\R^{n\times 2n}$ in canonical form whose left $n$ columns are the identity and whose row span is $\L$.
Multiplying the $i$th row of $N$ by $a_i$ for $i=1,\ldots,n$ yields
\begin{equation}\label{eq:matrix_row_op}
    \tilde{N} = 
   \begin{bmatrix}
       a_1  & \cdots & 0   & e_{11}  & \cdots & e_{1n} \\
       \vdots &  \ddots & \vdots  & \vdots & \ddots & \vdots \\
       0 & \cdots & a_n & e_{n1}  & \cdots & e_{nn}
   \end{bmatrix}.
\end{equation}
By fixing $\sign\left(p_{[1,\ldots,n]}\right)=\sigma_1\cdots\sigma_n$, 
we specify an orientation on the Pl\"ucker coordinates $\Pl(\L)$ and can thus talk about their signs.
Thus, if $\Pl(\L)$ satisfies condition~\eqref{it:1}, so do the minors of $\tilde{N}$.
Furthermore, the inequalities in assumptions~\eqref{it:2} and~\eqref{it:3} are rational in the Pl\"ucker coordinates and therefore are invariant under positive scaling, so they are also satisfied by the minors of $\tilde{N}$. 
Expanding these minors shows that the pair $\big(\ba,(e_{ij})_{i,j\in[n]}\big)$ satisfies the sign assignment $\sigma$ together with feasibility and stability. 
\end{proof}

\begin{definition}\label{def:partial_sign}
    Given a sign pattern $\sigma = \sign(\ba,\bB)$, we call the ordered sequence of signs in Item 1 of Theorem~\ref{prop:toalgebraicGrassmannian},
\begin{equation}\label{eq:partialsign}
    [\sigma_1\cdots \sigma_n]\sqcup \Biggl[\sigma_{ij}\prod_{k\in [n]\setminus i}\sigma_k \text{ for } i\in [n] \text{ and } j\in [n+1,\ldots, 2n]\Biggr]
\end{equation}
 the \emph{partial} sign assignment $\overline{\chi}(\sigma)$ of $\sigma$.
\end{definition}

We now turn to the theory of \emph{oriented matroids} to study the sign pattern of Pl\"ucker coordinates under the feasibility and stability conditions. 
This provides a combinatorial characterization of the realizable sign patterns of the Pl\"ucker coordinates.  

We recall that a \emph{matroid} is a pair $(E,\mathcal{B})$, in which $E$ is a finite set and $\mathcal{B}$ is a nonempty collection of subsets of $E$, called \emph{bases}, such that if $B_1,B_2$ are distinct bases and ${b_1\in B_1\setminus B_2}$ then there exists an element $b_2\in B_2\setminus B_1$ such that $(B_1\setminus \{b_1\}\cup \{b_2\})$ is a basis. This important notion generalizes the notion of linear independence among vectors. In \emph{oriented matroids}, bases come with signs: following \cite[Definition 3.5.2, Theorem 3.6.2]{valuated_matroids}, an oriented matroid of rank $k$ on $[N]$ is a matroid $([N],\mathcal{B})$ together with a nonzero, alternating map $\chi\colon[N]^k\rightarrow \{0,+,-\}$, called the \emph{chirotope}, satisfying the \emph{$3$-term Grassmann--Pl\"ucker relations} defined below.
If $M\in \mathbb{R}^{k\times N}$ has full rank, then the function $\chi$ can be defined to take the signs of the $k\times k$ minors. The oriented matroids obtained this way are called \emph{realizable}.
Thus any matrix $M$ representing a point of the real Grassmannian induces a chirotope. The ($3$-term) Grassmann--Pl\"ucker relations are: for any choice of distinct ${i_1,i_2,\ldots,i_n,j_1,j_2\in[2n]}$,
\begin{align}\label{ineq:3_term_gr_pl}
&\text{if } \chi(j_1, i_2, \ldots, i_n)\,
\chi(i_1, j_2, i_3, \ldots, i_n) \geq 0,\quad 
\chi(j_2, i_2, \ldots, i_n)\,
\chi(j_1, i_1, i_3, \ldots, i_n) \geq 0 \nonumber\\
&\text{then } \chi(i_1, i_2, \ldots, i_n)\,
\chi(j_1, j_2, i_3, \ldots, i_n) \geq 0. 
\end{align}
The Grassmann--Plücker relations propagate sign constraints: fixing the signs of certain minors induces the signs of others.  

In our setting the ground set is $[2n]$, and we take the chirotope induced by the ${n\times 2n}$ matrices in \eqref{eq:param_matrix}. Thus, oriented matroids serve as a combinatorial abstraction for analyzing possible sign patterns of Plücker coordinates and for detecting \emph{impossible} sign patterns. We summarize our discussion above:

\begin{corollary}\label{thm:comb_grass_gs}
Let $(\ba, \bB)\in\Rn\times \Rnn$ and let $\L=\rowspan(M)\in \Grnn$ be the associated linear space as defined in~\eqref{eq:param_matrix}.
Given an ordered sign pattern $\sigma\in\{-, +\}^{n+n^2}$, let $p=\mathrm{Pl}(\L)$ be the Pl\"ucker coordinates of $\L$, with orientation specified by \linebreak ${\sign(p_{[1,\ldots, n]}) = \sigma_1\cdots\sigma_n}$. 
If $(\ba, \bB)\in\mathcal{S}_n(\sigma)$, then the ordered sign patterns $\chi_S$ of the Pl\"ucker point $p$ are consistent with the positivity conditions from
\begin{enumerate}[leftmargin=*]
    \item the partial sign assignment $\overline{\chi}(\sigma)$ in \eqref{eq:partialsign},
    \item the feasibility inequalities in \eqref{eq:feasibility_global},
    \item the stability inequalities in \eqref{eq:stability_plucker_global},
    \item the oriented Grassmann--Pl\"ucker relations in \eqref{ineq:3_term_gr_pl}.
\end{enumerate}
\end{corollary}

We therefore formulate the \emph{realizability problem} in terms of the above positivity conditions to decide whether the ecology is possible.
  
\begin{definition}[Realizable sign assignment]\label{def:signcompletion}
A sign assignment $\chi$ on the Pl\"ucker coordinates is a potential completion of the partial sign assignment $\overline{\chi}(\sigma)$ if 
\begin{enumerate}[leftmargin=21pt]
    \item[(i)]  $\chi_S=\overline{\chi}(\sigma)_S$ for every $S$ whose sign is prescribed, and
    \item[(ii)] $\chi$ satisfies the positivity constraints from the feasibility inequalities~\eqref{eq:feasibility_global}, the stability inequalities~\eqref{eq:stability_plucker_global}, and the oriented Grassmann--Pl\"ucker relations~\eqref{ineq:3_term_gr_pl}.
\end{enumerate}
If no such potential completion exists for the sign data $\sigma$, we call the sign pattern $\sigma$ \emph{impossible}.
\end{definition}

To summarize, the Pl\"ucker map
\eqref{eq:parameter_pipeline} realizes the feasible stratum inside the Grassmannian.
On the other hand, the \emph{stability} conditions involve the specific scaling of the 
first $n$ columns of $M$ using the parameters $\ba$, so they do not define 
a subset of $\Grnn$ alone. Instead, they refine the feasible stratification 
on the lifted space of pairs $(\mathcal{L}, \ba)$. 
The feasible--stable stratification may thus be viewed as a refinement of the oriented matroid stratification of the real Grassmannians which imposes additional feasibility-stability conditions onto it. 
The resulting sign completion problem is a specialized oriented matroid \emph{realization problem} within this stratification. 
This view point connects the Lotka–Volterra dynamical system and its steady states to stratified semi-algebraic regions inside the Grassmannian, where they can be studied using tools of algebraic geometry.

\begin{example}\label{ex:n_3Continued} We revisit Example \ref{ex:n_3} of obligate mutualism of three species, now viewed inside the Grassmannian ${\rm Gr(3,6)}$. 
Consider the  matrix 
\begin{align}
    \label{eq:coupled_matroid_3_6}
M =   \begin{bmatrix}
a_1 & 0 & 0 & b_{1 1} & b_{1 2} & b_{1 3} \\
0 & a_2 & 0 & b_{2 1} & b_{2 2} & b_{2 3} \\
0 & 0 & a_3 & b_{3 1} & b_{3 2} & b_{3 3}
\end{bmatrix},
\end{align}
with the sign conditions that $a_i<0$, $b_{ii}>0$, and $b_{ij}<0$ for $i\neq j$. 
We denote this sign pattern by $\sigma$. 
Under the Pl\"ucker embedding, the coordinates of $M$ are
\begin{align*}\label{eq:PAB_N3relations}
& p_{124} = a_1 a_2 b_{31}, \,\, p_{145} = a_1 B_{13}, \,\,p_{125} = a_1 a_2 b_{32},\,\, p_{146} = a_1 B_{12},\,\, p_{126} = a_1 a_2 b_{33},\\ & p_{156} = a_1 B_{11}, \,\, 
p_{134} = - a_1 a_3 b_{21}, \,\, p_{245} = -a_2 B_{23},\,\, p_{135} = - a_1 a_3 b_{22}, \,\,  p_{246} = -a_2 B_{22}, \\ & p_{136} = - a_1 a_3b_{23}, \,\,
p_{256} = -a_2 B_{21}, \,\, p_{234} = a_2 a_3 b_{11}, \,\, p_{345} = a_3 B_{33}, \,\, p_{235} = a_2 a_3 b_{12}, \\ 
& p_{346} = a_3 B_{32},\,\, p_{236} = a_2 a_3 b_{13}\,\,  p_{356} = a_3 B_{31},\,\, p_{123} = a_1 a_2 a_3,\,\, p_{456} = \det(\bB),
\end{align*}
where $B_{ij}$ denotes the determinant of the matrix obtained by deleting row $i$ and column $j$ of $\bB$ deleted. The assumed signs of $a_i$ and $b_{ij}$ imply the following partial sign assignment $\overline{\chi}(\sigma)$:

\begin{equation*}
\begin{tabular}{cccccccccc}
\toprule
$\chi_{123}$ & $\chi_{124}$ & $\chi_{125}$ & $\chi_{126}$ & $\chi_{134}$ & $\chi_{135}$ & $\chi_{136}$ & $\chi_{145}$ & $\chi_{146}$ & $\chi_{156}$ \\
$-$ & $-$ & $-$ & $+$ & $+$ & $-$ & $+$ & $-$ & $+$ & $?$ \\
$\chi_{234}$ & $\chi_{235}$ & $\chi_{236}$ & $\chi_{245}$ & $\chi_{246}$ & $\chi_{256}$ & $\chi_{345}$ & $\chi_{346}$ & $\chi_{356}$ & $\chi_{456}$ \\
$+$ & $-$ & $-$ & $-$ & $?$ & $-$ & $?$ & $+$ & $-$ & $?$ \\
\bottomrule
\end{tabular}
\end{equation*}
The unknown entries (presented with ``?" in the table) correspond to principal $2\times2$ minors and to $\det(\bB)$.
From stability, we require $\chi_{456}=+$.
Combining this with the oriented Grassmann--Pl\"ucker relations~\eqref{ineq:3_term_gr_pl}, we infer the signs $\chi_{156}= +$, $\chi_{246} = -$, $\chi_{345} = -$. 
The three feasibility inequalities are expressed as
\begin{align*}
   &a_1(-b_{23}b_{32}+b_{22}b_{33})+a_2(b_{13}b_{32}-b_{12}b_{33})+a_3(-b_{13}b_{22}+b_{12}b_{23})= p_{156} + p_{256} + p_{356} > 0,\\
    &a_1(b_{23}b_{31}-b_{21}b_{33})+a_2(-b_{13}b_{31}+b_{11}b_{33})+a_3(b_{13}b_{21}-b_{11}b_{22}) = -p_{146} - p_{246} - p_{346} > 0,\\
    & a_1(-b_{22}b_{31}+b_{21}b_{32})+a_2(b_{12}b_{31}-b_{11}b_{32})+a_3(-b_{12}b_{21}+b_{11}b_{22}) = p_{145} + p_{245} + p_{345} >0. 
\end{align*}
The last line contradicts $\chi_{145}=\chi_{245}=\chi_{345}=-$, showing that this sign configuration is impossible.
\end{example} 

We conclude this section with a few remarks on potential connections to other algebraic frameworks. In the present work, we consider the map into the Grassmannian given by the maximal minors of the matrix $[\diag(\ba) \mid \bB]$. However, it would also be natural to study ecological patterns through the lens of the \emph{collineation variety}, defined via the map taking all minors of a matrix of each size \cite{Gesmundo2025}. In fact, the signs of the minors of smaller size play a crucial role in the sign analysis in Example~\ref{ex:n_3Continued}. This extended embedding captures additional multilinear and combinatorial structures beyond those encoded by the Grassmannian, and may provide a complementary geometric framework for understanding coexistence phenomena and their obstructions.

\section{Effective Methods for Sign Pattern Completion}
\label{sec:computation}
The theoretical framework developed in Sections~\ref{sec:RouthHurwitz} and~\ref{sec:geometry} provides the mathematical foundation for determining feasibility and stability. Given a sign pattern $\sigma$ on the parameters $(\ba, \bB)$, we are interested in whether the feasible-stable stratum $S_n(\sigma)$ defined in~\eqref{eq:feas_stable_stratum} is nonempty. From $\sigma$, we obtain a partial sign pattern $\overline{\chi}(\sigma)$ on the Plücker coordinates through the matrix representation $M = [\diag(\ba) | \bB]$, as in Definition~\ref{def:partial_sign}. The central computational challenge is to complete this partial sign pattern efficiently, i.e., determining which full sign assignments are compatible with feasibility and stability. This section presents our computational approach to this sign pattern completion through a two-pronged algorithmic strategy of \texttt{Propagation} and \texttt{Branching with Ecological Constraints}. 
In \texttt{Propagation}, we systematically apply Grassmann--Plücker relations to infer unknown signs $\chi(\sigma)$ until no further propagation is possible. In \texttt{Branching with Ecological Constraints}, we branch on remaining unknowns while naturally incorporating feasibility and stability checks to prune the search space. 

We begin by recalling the problem we are interested in, which we discussed through the lens of the semialgebraic feasible-stable strata in Section~\ref{sec:RouthHurwitz} and the Grassmannian in Section~\ref{sec:geometry}. In this section, we approach the problem~\ref{prob:feasible_ecology} from a purely computational perspective. As introduced in Definition~\ref{def:partial_sign}, given an ordered sign pattern $\sigma$ on the ecological parameters $(\ba, \bB)$, we obtain an ordered partial sign assignment $\overline{\chi}(\sigma)$  on the $\binom{2n}{n}$ Plücker coordinates of the real Grassmannian $\Grnn$ through the pipeline outlined in~\eqref{eq:parameter_pipeline}. These correspond to sign patterns attained in the complement of the arrangement of coordinate hyperplanes $p_{S} = 0$ in $\mathbb{P}^{\binom{2n}{n}-1}$ for $S\in\binom{[2n]}{n}$ that stratify the Grassmannian, with each stratum representing a different combinatorial type of sign pattern. Some Plücker coordinates have their signs determined by the parameter constraints, while others signs remain unknown. For computational purposes, we extend our earlier notation from Definition~\ref{def:partial_sign} by representing this ordered partial sign pattern as a function on $n$-element bases,
\begin{equation}\label{eq:sign_assignement}
    \chi\colon\baseset \to \{?, - , +\},\quad\chi(S)=\left\{\begin{array}{cl}
    \overline{\chi}(\sigma)&\text{ if the sign is known,}\\
    ?&\text{otherwise}.
    \end{array}\right.
\end{equation}
The computational challenge can be naturally viewed as navigating a binary decision tree. Starting from this partial sign pattern, we must decide whether each unknown basis should be assigned $+$ or $-$. Without additional constraints, this yields a complete binary tree with $2^k$ leaf nodes, where $k$ is the number of unknown signs. However, the geometric structure of the Grassmannian provides powerful constraints through the Grassmann--Plücker relations~\eqref{ineq:3_term_gr_pl}, transforming this exponential search into a constrained branch propagation problem.

\begin{problem}[Sign Pattern Completion]\label{prob:signcompletion}
Given the sign pattern $\sigma$ for the matrix\linebreak ${M = [\diag(\ba) | \bB] \in \R^{n \times 2n}}$ from the parameter matrix representation~\eqref{eq:param_matrix}, we obtain a partial sign pattern $\overline{\chi}(\sigma)$ on Plücker coordinates. Compute all possible completions of the sign pattern having no obstruction to the positivity conditions from the Grassmann--Pl\"ucker relations and the feasibility and stability conditions. 
\end{problem}

When we assign a sign to one basis, the three-term relations~\eqref{ineq:3_term_gr_pl} often force specific signs on other bases, eliminating entire subtrees from consideration. Moreover, incompatible sign assignments lead to contradictions that can be detected early, allowing us to backtrack without exploring deeper branches. As announced above, our algorithmic strategy for Problem~\ref{prob:signcompletion} divides into two parts: verifying consistency with the Grassmann–Plücker relations and checking the feasibility-stability conditions. In fact, \texttt{Propagation} systematically applies Grassmann--Plücker relations to infer as many unknown signs as possible. For each unknown basis, we generate compatible bases and apply the 3-term relations. If a relation uniquely determines a sign, then we update the pattern and continue; if a contradiction arises, then we prune the branch generating the invalid assignment. \texttt{Propagation} terminates when no further inferences are possible. \texttt{Branching with Ecological Constraints} activates when propagation terminates with unknown signs remaining in the partial sign pattern. We select undetermined bases and recursively explore both possible sign assignments, with each branch triggering new \texttt {Propagation}. Complete sign patterns are validated against the feasibility and stability constraints, providing natural pruning of branches that are not feasible and stable. The algorithms described in this section are implemented in \textsc{Maple} and available in the repository \cite{gitlab}.

For \texttt{Propagation}, our algorithm consists of three core mathematical operations, which we formulate as algorithmic steps as presented in Algorithm~\ref{alg:propagate_sign}. Step 1 is the \texttt{Basis Selection}: For each unknown basis $B_{\text{u}}$, systematically examine all possible 2-element subsets $Y \subseteq B_{\text{u}}$. Step 2 is the \texttt{Compatible Basis Generation}: For each subset $Y$, find all known bases $B_{\text{c}}$ that are compatible with $B_{\text{u}}$ via $Y$ (i.e., $B_{\text{u}}\setminus Y \subseteq B_{\text{c}}$ and $Y \cap B_{\text{c}} = \emptyset$), then generate the five auxiliary bases needed for the 3-term Grassmann--Plücker relation. Step 3 is the \texttt{Sign Inference}: Apply the 3-term Grassmann--Plücker relation to infer the sign of $B_{\text{u}}$ or detect conflicts.

\begin{algorithm}[h!]
    \caption{\textbf{PropagateSignForBasis}($\chi$, $B_{\text{u}}$)}
    \label{alg:propagate_sign}
    \KwIn{$\chi$ -- current sign pattern, $B_{\text{u}}$ -- unknown basis to infer}
    \KwOut{inferred sign, conflict, or no inference}
    \SetAlgoLined
    \DontPrintSemicolon
    $n \leftarrow |B_{\text{u}}|$\;
    \ForEach{2-element subset $Y = [y_1, y_2] \subseteq B_{\mathrm{u}}$}{
        $B_{\setminus Y} \leftarrow B_{\text{u}} \setminus Y$ \tcp*{Common $n-2$ elements}
        \;
        \ForEach{known basis $B_{\mathrm{c}}$ in $\chi$}{
            \If{$B_{\setminus Y} \subseteq B_{\mathrm{c}}$ and $Y \cap B_{\mathrm{c}} = \emptyset$}{
                \tcp{$B_{\text{c}}$ is compatible with $B_{\text{u}}$ via $Y$}
                \tcp{Extract the complement set from the known basis}
                $X \leftarrow B_{\text{c}} \setminus B_{\setminus Y} $ where $X = [x_1, x_2]$\;
                \;
                \tcp{Construct four auxiliary bases}
                $J_1 \leftarrow B_{\setminus Y}  \cup [y_1, x_2]$\;
                $J_2 \leftarrow B_{\setminus Y}  \cup [x_1, y_2]$\;
                $J_3 \leftarrow B_{\setminus Y}  \cup [y_2, x_2]$\;
                $J_4 \leftarrow B_{\setminus Y}  \cup [y_1, x_1]$\;
                \;
                \tcp{Apply 3-term Grassmann--Plücker relation}
                \If{ $\chi(B_{\mathrm{c}}), \chi(J_1), \chi(J_2), \chi(J_3), \chi(J_4) \neq ?$}{
                    Try to infer $\chi(B_{\text{u}})$ by solving G-P relation\;
                    Record this inference with evidence\;
                }
            }
        }
    }
    \;
    \Return aggregated inference result\;
\end{algorithm}

If \texttt{Propagation} terminates with unknown signs remaining in the partial sign pattern, then the algorithm transitions to \texttt{Branching with Ecological Constraints}. At this point, the Grassmann--Plücker relations have extracted all possible combinatorial information, and further progress requires making assumptions on the remaining undetermined signs and verifying them. This transition gives a natural opportunity to check the ecological constraints given by the feasibility and stability conditions and to prune potential branches that cannot represent feasible and stable sign patterns. In the base case, if all bases have determined signs, then the algorithm validates the complete pattern against the positivity of the feasibility constraints in~\eqref{eq:feasibility_global} and the positivity of the stability constraints in~\eqref{eq:stability_plucker_global}. We note here that our current implementation does not incorporate the higher-order positivity conditions imposed by the Hurwitz polynomials introduced in~\eqref{eq:Hurwitzdeterminant} as these conditions are no longer linear in the Pl\"ucker coordinates and thus are more costly to verify.

The complete search procedure is a simple recursive wrapper that collects all valid  completions. If no sign can be propagated, we pick an unknown basis in the current partial sign pattern and generate a branch for either possible sign assignments $\{-, +\}$. 
In each branch, we create an augmented copy of the current sign pattern with the new sign assignment and invoke the \texttt{Propagation} again. If the procedure runs without detecting any combinatorial conflicts, the algorithm recurses on the updated pattern. However, if \texttt{Propagation} reveals a contradiction [indicating that the sign pattern cannot be realized in any stratum of $\Grnn$], then the branch is pruned. 

We close this section by analyzing the convergence properties and computational complexity of our algorithm. Understanding these theoretical foundations provides confidence in the practical applicability of our approach to sign pattern completion problems. Let $k = |\{B\mid \chi(B) = \mbox{?}\}|$ be the number of initial unknowns and $N = \binom{2n}{n}$ the total number of bases. The two-part structure affects complexity analysis as follows: In \texttt{Propagation}, each step examines all unknown bases and applies the Grassmann--Plücker relations, requiring $O(N^2)$ operations per propagation round. In the best case, propagation can determine all remaining signs, eliminating the need for branching entirely. In \texttt{Branching}, the worst-case time complexity is $O(2^k \cdot N^2)$, corresponding to exploring all $2^k$ possible sign assignments with $O(N^2)$ propagation cost per node. In practice, however, the alternation between branching and propagation significantly reduces the effective search space by inferring signs, reducing $k$ by a factor $p \in (0, 1)$ and yielding an effective complexity of $O(2^{k(1-p)} \cdot N^2)$. The space complexity is $O(k \cdot N)$ for maintaining the recursion stack and sign pattern copies.

The convergence of our algorithm is guaranteed by the finite nature of the search space and the monotonic progress of sign assignments in both phases. \texttt{Propagation} terminates when no further inferences are possible, and \texttt{Branching} makes explicit progress by assigning signs to unknowns. The following theorem formalizes this observation:

\begin{theorem}[Termination of two-part Algorithm]
	\label{thm:termination}
	The two-phase algorithm terminates in finite time for any initial partial sign pattern $\overline{\chi}$ on the Plücker coordinates of the real Grassmannian $\Grnn$.
\end{theorem}

\begin{proof} The search space consists of the $N = \binom{2n}{n}$ bases of $n$ elements, hence is finite.  
Let $U$ be the set of bases with undetermined signs. In \texttt{Propagation}, Grassmann–Plücker relations can determine at most $|U|$ new signs, and each inference strictly reduces $|U|$.  
In \texttt{Branching}, every recursive call assigns a sign to at least one unknown basis, again strictly reducing $|U|$.  
Thus, along any path of the search tree, $|U|$ decreases monotonically from its initial value $k = |U|_{\mathrm{initial}}$ to $0$, giving a recursion depth of at most $k$. Since each step either infers a new sign or prunes a branch via a detected conflict, no branch can continue indefinitely.  
Because both $N$ and $k$ are finite, the entire search terminates in finite time.
\end{proof}

\section{Experiments}
\label{Sec:Experim}
To visualize the experiments that follow, we adopt the approach of \cite{Pierre} and encode networks of types of ecological interactions (which we will also call \emph{ecological patterns}) as colored complete graphs on $n$ vertices, as in Figure~\ref{fig:Pierregraph}).
We recall that, in \cite{Pierre}, the authors classify impossible ecologies for for $n<3$ via exhaustive case-by-case computations and extend this to conjectures for $n = 4,5$ by random numerical sampling of feasible and stable steady states of the Lotka--Volterra system.

\begin{figure}[b!]
\centering
\tikzset{
  fillednode/.style={circle, draw=black, fill=black, inner sep=2pt},
  node/.style={circle, draw=black, fill=white, inner sep=2pt},
  edge/.style={line width=1pt},
  redge/.style={edge, red},
  bedge/.style={edge, blue}
}
\noindent\begin{tikzpicture}[>=stealth,scale=1.3]
\draw[postaction={decorate},decoration={markings,mark=at position 0.50 with {\arrow{<}}},thick] (0.3090,0.9511) -- (-0.8090,0.5878);
\draw[postaction={decorate},decoration={markings,mark=at position 0.80 with {\arrow{<}}},thick] (0.3090,0.9511) -- (-0.8090,-0.5878);
\draw[red,thick] (0.3090,0.9511) -- (0.3090,-0.9511);
\draw[red,thick] (0.3090,0.9511) -- (1.0000,-0.0000);
\draw[blue,thick] (-0.8090,0.5878) -- (-0.8090,-0.5878);
\draw[postaction={decorate},decoration={markings,mark=at position 0.20 with {\arrow{>}}},thick] (-0.8090,0.5878) -- (0.3090,-0.9511);
\draw[postaction={decorate},decoration={markings,mark=at position 0.20 with {\arrow{>}}},thick] (-0.8090,0.5878) -- (1.0000,-0.0000);
\draw[postaction={decorate},decoration={markings,mark=at position 0.50 with {\arrow{>}}},thick] (-0.8090,-0.5878) -- (0.3090,-0.9511);
\draw[postaction={decorate},decoration={markings,mark=at position 0.20 with {\arrow{>}}},thick] (-0.8090,-0.5878) -- (1.0000,-0.0000);
\draw[blue,thick] (0.3090,-0.9511) -- (1.0000,-0.0000);
\fill (0.3090,0.9511) circle (0.10);
\fill (-0.8090,0.5878) circle (0.10);
\fill (-0.8090,-0.5878) circle (0.10);
\filldraw[fill=white] (0.3090,-0.9511) circle (0.10);
\filldraw[fill=white] (1.0000,-0.0000) circle (0.10);
\end{tikzpicture}
\vspace{.5cm}
\hspace{1cm}
\noindent\begin{tikzpicture}[every text node part/.style={align=center},line cap=round,>=stealth,font=\small,scale=1.3]
\node at (-0.08,-1.5) {$\vec{\dot{x}}=\diag(\vec{x})\left(\ba-\bB\cdot\vec{x}\right)$};
\draw[<-] (-1.05,-1.65) -- (-1.05,-2.2) -- (-0.5,-2.2) node[anchor=west,inner sep=1.5pt] {$\displaystyle\dot{x}_i=x_i\left(a_i-\sum_{j=1}^n{b_{ij}x_j}\right),\;i=1,2,\dots,N$};
\draw[<-] (0.9,-2.4) -- (0.9,-2.9) -- (-0.2,-2.9) -- (-0.2,-3.2) node[anchor=north] {growth rates\\$a_i\gtrless 0$};
\draw[<-] (2.,-2.4) -- (2.,-3.2) node[anchor=north] {interactions\\$b_{ij}\gtrless 0,b_{ii}>0$};
\end{tikzpicture}


\noindent\begin{tikzpicture}[every text node part/.style={align=center},line cap=round,>=stealth,font=\small,scale=1.3]
\begin{scope}[shift={(0,0.2)}]
\filldraw (0,-2) circle(2.5pt);
\node at (-0.25,-2) {$i\vphantom{j}$};
\node[anchor=west] at (0.4,-2) {growing species:\vphantom{d}};
\node at (1.7,-2.4) {$a_i > 0 $};
\node at (-0.25,-2.9) {$i\vphantom{j}$};
\draw (0,-2.9) circle (2.5pt);
\node[anchor=west] at (0.4,-2.9) {dying species:};
\node at(1.7,-3.3) {$a_i < 0 $};
\end{scope}
\begin{scope}[shift={(3.3,2)}]
\draw[red,thick] (0.1,-3.8) node[anchor=east,inner sep=1.5pt,black] {$i\vphantom{j}$} -- (0.6,-3.8) node[anchor=west,inner sep=1.5pt,black] {$j$};
\draw[blue,thick] (0.1,-4.5) node[anchor=east,inner sep=1.5pt,black] {$i\vphantom{j}$} -- (0.6,-4.5) node[anchor=west,inner sep=1.5pt,black] {$j$};
\draw[postaction={decorate},decoration={markings,mark=at position 0.6 with {\arrow{>}}},thick] (0.1,-5.2) node[anchor=east,inner sep=1.5pt,black] {$i\vphantom{j}$} -- (0.6,-5.2) node[anchor=west,inner sep=1.5pt,black] {$j$};
\node[anchor=west] at (1,-3.8){competition:};
\node at (3.6,-3.8) {$b_{ij},b_{ji}>0$};
\node[anchor=west] at (1,-4.5) {mutualism:\vphantom{p}};
\node at (3.6,-4.5){$b_{ij},b_{ji} < 0$};
\node[anchor=west] at (1,-5.2) {predation:};
\node at (3.6,-5.2){$b_{ij} < 0,b_{ji}>0$};
\end{scope}
\end{tikzpicture}

\caption{Lotka--Volterra ecological dynamics on a network of $n=5$ species: definition of the different types of ecological interactions. Figure redrawn from \cite{Pierre}.}
\label{fig:Pierregraph}
\end{figure}
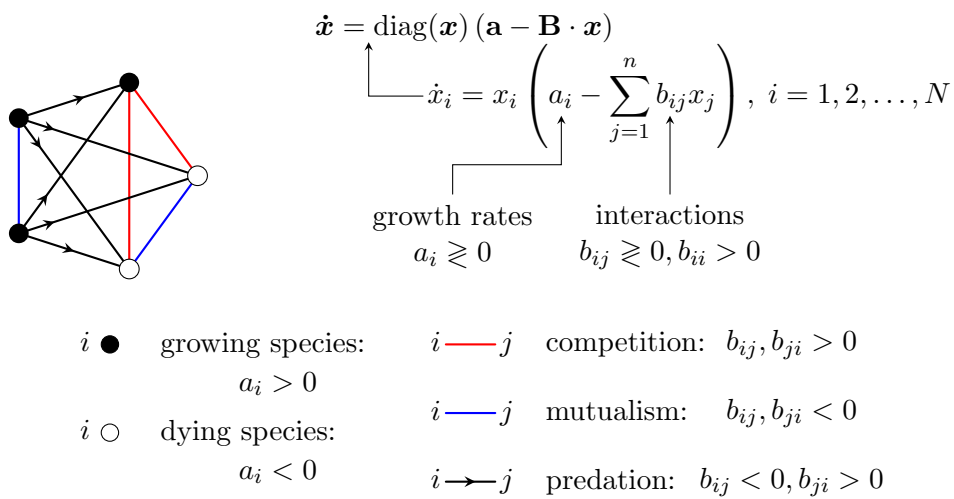

Our aim is to complement and refine this approach by using tools from computational algebraic geometry introduced above, enabling an exact symbolic analysis of the parameter space governing feasibility and stability. In particular, we seek to characterize the semialgebraic structure of the coexistence regions associated with given interaction topologies or ecological patterns. The following theorem summarizes our results:

\begin{theorem}\label{thm:impossible4} The three ecological networks highlighted in the boxed region of Figure~\ref{fig:n4impossible} are impossible ecologies, i.e., no choice of growth rates $\ba$ and interaction strengths $\bB$ with the prescribed signs leads to a positive, locally asymptotically stable equilibrium of the Lotka--Volterra system, at which all four species coexist at positive abundances.
\end{theorem}

\begin{proof}
Suppose, to the contrary, that such an equilibrium exists for one of the three networks. Then there would be parameters $(\ba,\bB)$ consistent with the sign pattern and a Pl\"ucker point $p$ whose sign assignment $\chi(p)$ extends the partial assignment $\overline{\chi}(\sigma)$ and satisfies the feasibility inequalities~\eqref{eq:feasibility_global}, the stability inequalities~\eqref{eq:stability_plucker_global}, and the Grassmann--Pl\"ucker relations~\eqref{ineq:3_term_gr_pl}, i.e. a valid sign completion~\eqref{eq:sign_assignement}. Our computational check, using the two-pronged algorithm described in Section~\ref{sec:computation}, exhaustively searched for such completions and finds none in each of the three cases, certifying that the feasible-stable set is empty. This contradiction proves the claim.
\end{proof}\\

After the impossible cases revealed by Theorem~\ref{thm:impossible4}, we turn to configurations for which our computations produce few completions which still provide useful information.
For $n = 3$, it was shown in~\cite{Pierre} that exactly four ecological sign patterns are impossible. 
Two of these four configurations are symmetric, and our algorithm correctly certifies them to have no valid sign completions, i.e., to be impossible. For the remaining two (asymmetric) configurations, the algorithm returns a single completion each.  However, these completions violate additional sign constraints involving smaller minors that are not yet incorporated into our implementation, see Example~\ref{Ex:N3Continued_second}. Hence, all four ecologies identified in~\cite{Pierre} remain impossible.

In addition to the three $n = 4$ ecological patterns that our code proves to be impossible, we observe a striking difference in the completion counts for the remaining two symmetric patterns also claimed to be impossible in~\cite{Pierre} (Figure~\ref{fig:n4impossible}). While the code returns only 64 completions for each of these symmetric cases, the corresponding asymmetric patterns yield over 1000 completions. Even when stability and feasibility checks are disabled, and only the chirotope rule is applied, the symmetric ecologies yield only 256 completions, substantially fewer than the generic case. This demonstrates that the chirotope rule alone imposes strong combinatorial restrictions on the sign patterns of Plücker coordinates for symmetric interaction structures. These observations suggest that symmetry in ecological networks introduces intrinsic combinatorial obstructions to coexistence. We anticipate that further analysis, such as incorporating additional stability sign conditions, or examining signs of smaller minors, could rule out these cases completely.

\begin{figure}[t]
\centering

\tikzset{
  fillednode/.style={circle, draw=black, fill=black, inner sep=2pt},
  node/.style={circle, draw=black, fill=white, inner sep=2pt},
  edge/.style={line width=1pt},
  redge/.style={edge, red},
  bedge/.style={edge, blue},>=stealth
}

\noindent
\makebox[\textwidth][c]{%
  \fbox{%
    \begin{minipage}[t]{0.5\textwidth}
    \centering
    \foreach \i in {0,1,2} {
      \begin{tikzpicture}[scale=0.8, baseline={(current bounding box.center)}]
        \coordinate (A) at (0,2);
        \coordinate (B) at (-1,1);
        \coordinate (C) at (1,1);
        \coordinate (D) at (0,0);

        \ifnum\i=0
          \draw[bedge] (A) -- (B);
          \draw[bedge] (A) -- (C);
          \draw[bedge] (A) -- (D);
          \draw[bedge] (B) -- (D);
          \draw[bedge] (B) -- (C);
          \draw[bedge] (C) -- (D);
        \else\ifnum\i=1
          \draw[redge] (A) -- (B);
          \draw[redge] (A) -- (C);
          \draw[redge] (A) -- (D);
          \draw[bedge] (B) -- (D);
          \draw[bedge] (B) -- (C);
          \draw[bedge] (C) -- (D);
        \else \ifnum \i=2
          \draw[bedge] (A) -- (B);
          \draw[redge] (A) -- (C);
          \draw[redge] (A) -- (D);
          \draw[redge] (B) -- (D);
          \draw[redge] (B) -- (C);
          \draw[bedge] (C) -- (D);
        \fi\fi\fi

        \ifnum\i=0
          \node[node] at (A) {};
          \node[node] at (B) {};
          \node[node] at (C) {};
          \node[node] at (D) {};
        \else \ifnum\i=1
          \node[fillednode] at (A) {};
          \node[node] at (B) {};
          \node[node] at (C) {};
          \node[node] at (D) {};
        \else \ifnum \i=2
          \node[fillednode] at (A) {};
          \node[fillednode] at (B) {};
          \node[node] at (C) {};
          \node[node] at (D) {};
        \fi\fi\fi
      \end{tikzpicture}
      \hspace{0.2cm}
    }
    \end{minipage}%
  }

  \raisebox{0pt}{ 
  \begin{minipage}[t]{0.4\textwidth}
    \centering
    \foreach \i in {3,4} {
      \begin{tikzpicture}[scale=0.8, baseline={(current bounding box.center)}]
        \coordinate (A) at (0,2);
        \coordinate (B) at (-1,1);
        \coordinate (C) at (1,1);
        \coordinate (D) at (0,0);

        \draw[bedge] (A) -- (B);
        \draw[redge] (A) -- (C);
        \draw[redge] (A) -- (D);
        \draw[redge] (B) -- (D);
        \draw[redge] (B) -- (C);
        \draw[bedge] (C) -- (D);

        \ifnum\i=3
          \node[fillednode] at (A) {};
          \node[node] at (B) {};
          \node[node] at (C) {};
          \node[node] at (D) {};
        \else
          \node[node] at (A) {};
          \node[node] at (B) {};
          \node[node] at (C) {};
          \node[node] at (D) {};
        \fi
      \end{tikzpicture}
      \hspace{0.2cm}
    }
  \end{minipage}%
  }
}

\vspace{0.5cm}
\foreach \i in {0,...,5} {
  \begin{tikzpicture}[scale=0.8]
    \coordinate (A) at (0,1);
    \coordinate (B) at (-1,0);
    \coordinate (C) at (1,0);
    \coordinate (D) at (0,-1);

    \ifnum\i=0
      \draw[postaction={decorate},decoration={markings,mark=at position 0.50 with {\arrow{>}}},thick] (B) -- (A);
      \draw[postaction={decorate},decoration={markings,mark=at position 0.40 with {\arrow{>}}},thick] (B) -- (C);
      \draw[midarrow, thick] (B) -- (D);
      \draw[redge] (A) -- (D);
      \draw[redge] (A) -- (C);
      \draw[bedge] (C) -- (D);
    \else\ifnum\i=1
      \draw[bedge] (A) -- (B);
      \draw[midarrow, thick] (A) -- (C);
      \draw[midarrow, thick] (A) -- (D);
      \draw[midarrow, thick] (B) -- (C);
      \draw[midarrow, thick] (B) -- (D);
      \draw[bedge] (C) -- (D);
    \else\ifnum\i=2
      \draw[redge] (B) -- (A);
      \draw[midarrow, thick] (C) -- (A);
      \draw[redge] (D) -- (A);
      \draw[redge] (B) -- (C);
      \draw[bedge] (B) -- (D);
      \draw[redge] (C) -- (D);
    \else\ifnum\i=3
      \draw[redge] (A) -- (B);
      \draw[redge] (A) -- (C);
      \draw[redge] (A) -- (D);
      \draw[midarrow, thick] (B) -- (D);
      \draw[midarrow, thick] (C) -- (B);
      \draw[midarrow, thick] (D) -- (C);
    \else\ifnum\i=4
    \draw[midarrow, thick] (A) -- (C);
      \draw[midarrow, thick] (C) -- (B);
      \draw[midarrow, thick] (D) -- (C);
      \draw[midarrow, thick] (B) -- (D);
      \draw[midarrow, thick] (A) -- (D);
      \draw[midarrow, thick] (A) -- (B);
    \else \i=5
       \draw[midarrow, thick] (A) -- (B);
      \draw[midarrow, thick] (C) -- (A);
      \draw[redge] (A) -- (D);
      \draw[midarrow, thick] (B) -- (D);
      \draw[redge] (B) -- (C);
      \draw[midarrow, thick] (D) -- (C);
    \fi\fi\fi\fi\fi

    \ifnum\i=0
      \node[fillednode] at (A) {};
      \node[fillednode] at (B) {};
      \node[node] at (C) {};
      \node[node] at (D) {};
    \else \ifnum\i=1
      \node[fillednode] at (A) {};
      \node[fillednode] at (B) {};
      \node[node] at (C) {};
      \node[node] at (D) {};
      \else \ifnum\i=2
     \node[fillednode] at (A) {};
      \node[node] at (B) {};
      \node[node] at (C) {};
      \node[node] at (D) {};
     \else \ifnum\i=3
       \node[fillednode] at (A) {};
      \node[node] at (B) {};
      \node[node] at (C) {};
      \node[node] at (D) {};
     \else \ifnum\i=4
      \node[fillednode] at (A) {};
      \node[node] at (B) {};
      \node[node] at (C) {};
      \node[node] at (D) {};
    \else \i=5
       \node[node] at (A) {};
      \node[node] at (B) {};
      \node[node] at (C) {};
      \node[node] at (D) {};
    \fi \fi \fi \fi \fi
\end{tikzpicture}
  \hspace{0.6cm}
}
\caption{Ecological patterns of $n=4$ species conjectured to be impossible in~\cite{Pierre}. The ecological patterns in the first row are symmetric; those in the second row are asymmetric. Our computations establish that the symmetric patterns contained in the box are impossible. }
\label{fig:n4impossible}
\end{figure}
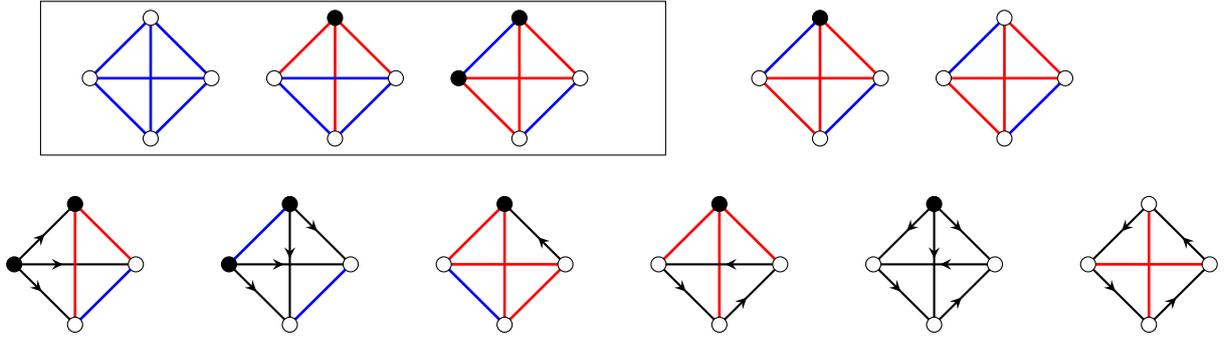

\begin{example}
\label{Ex:N3Continued_second} 
We consider \emph{facultative predation on two obligate mutualists} with $n=3$, shown to be an impossible ecology in \cite{Pierre}. This corresponds to the sign pattern $a_1>0$, $a_2<0$, $a_3<0$ and $b_{11}$, $b_{22}$, $b_{33}>0$ and $b_{21}$, $b_{31}>0$ and $b_{12}$, $b_{13}$, $b_{23}$, $b_{32}<0$. We thus have the ordered sign assignment $\sigma=(+,-,-,+,-,-,+,+,-,+,-,+)$. Our computation shows the following sign pattern $\chi$ is the unique valid completion of $\sigma$:

\begin{equation*}
\begin{tabular}{cccccccccc}
\toprule
$\chi_{123}$ & $\chi_{124}$ & $\chi_{125}$ & $\chi_{126}$ & $\chi_{134}$ & $\chi_{135}$ & $\chi_{136}$ & $\chi_{145}$ & $\chi_{146}$ & $\chi_{156}$ \\
$+$ & $-$ & $+$ & $-$ & $+$ & $+$ & $-$ & $-$ & $+$ & $+$ \\
$\chi_{234}$ & $\chi_{235}$ & $\chi_{236}$ & $\chi_{245}$ & $\chi_{246}$ & $\chi_{256}$ & $\chi_{345}$ & $\chi_{346}$ & $\chi_{356}$ & $\chi_{456}$ \\
$+$ & $-$ & $-$ & $+$ & $+$ & $-$ & $-$ & $-$ & $-$ & $+$ \\
\bottomrule
\end{tabular}
\end{equation*}
The above sign pattern does not violate the constraints imposed by the feasibility and stability conditions in Pl\"ucker coordinates. 
We now examine the feasibility constraints \eqref{eq:feasibility_global} but rewrite the coordinates using Grassmann--Pl\"ucker relations,
\begin{equation*}\label{eq:reformFeasPluckRel1}
\begin{aligned}
    &p_{156} + p_{256} + p_{356} 
\\&\qquad=\frac{1}{p_{123}}(-p_{135}p_{126}+p_{136}p_{125}-p_{235}p_{126}+p_{236}p_{125}-p_{235}p_{136}+p_{135}p_{236})
\\&\qquad=\frac{1}{p_{123}}[(-p_{135}p_{126}+p_{135}p_{236})+(p_{136}p_{125}-p_{235}p_{126}+p_{236}p_{125}-p_{235}p_{136})]
\\&\qquad=  \frac{1}{p_{123}}[p_{135}(-p_{126}+p_{236})+(p_{136}p_{125}-p_{235}p_{126}+p_{236}p_{125}-p_{235}p_{136})]>0,
\end{aligned}
\end{equation*}
\begin{equation*}\label{eq:reformFeasPluckRel2}
    \begin{aligned}
        &p_{146} + p_{246} + p_{346} \\
&\qquad=\frac{1}{p_{123}}[p_{134}(-p_{126}+p_{236})+(p_{124}p_{136}-p_{234}p_{126}+p_{124}p_{236}-p_{234}p_{136})]<0.
    \end{aligned}
\end{equation*}
Under the sign assignments of $\sigma$, we have ${p_{136}p_{125}-p_{235}p_{126}+p_{236}p_{125}-p_{235}p_{136}<0}$, $p_{124}p_{136}-p_{234}p_{126}+p_{124}p_{236}-p_{234}p_{136}>0$,
so the above implies $-p_{126}+p_{236}>0$ and ${-p_{126}+p_{236}<0}$, a contradiction. Hence this ecology is impossible. Similar computations show that \emph{obligate cyclic predation} \cite{Pierre}, corresponding to the sign pattern\linebreak $\sigma=(-,-,-,+,-,+,+,+,-,-,+,+)$, is impossible. Together with the case of obligate mutualism in Example~\ref{ex:n_3} and \emph{competition with two obligate mutualists}~\cite{Pierre}, with the sign pattern $\sigma=(+,-,-,+,+,+,+,+,-,+,-,+)$, that are directly certified to be impossible by our code, these are the only four impossible ecologies for $n=3$. 
\end{example}


\section{Computations via \texttt{HypersurfaceRegions.jl}}
\label{Sec:HypersurfacesComp}

In this section, we study feasibility and stability by computing regions in the complement of a hypersurface arrangement. As discussed in Section~\ref{sec:RouthHurwitz}, the feasibility and stability conditions of the Lotka--Volterra system are polynomial inequalities in the parameters $(\ba, \bB)\in \Rn\times \Rnn.$ The feasibility conditions require $\feascond>0$, while stability is determined by the positivity of the Routh--Hurwitz polynomials. We consider the hypersurface arrangement that consists of the hypersurfaces defined by these polynomials i.e., the zero sets of the polynomials, together with the coordinate hyperplanes given by $a_i = 0, b_{ij}=0$ for $i,j\in [n]$. 

In~\cite{cummings2024smooth}, the authors introduced an algorithm that combines Morse theory with numerical solution of ODEs to compute the number of connected regions in a hypersurface arrangement, along with their Euler characteristics and a representative point from each region. This algorithm was later adapted and implemented in Julia as the package \texttt{HypersurfaceRegions.jl}~\cite{breiding2025computing}. 
We use this package to compute possible ecologies for $n=2,3,4$. Related work using methods from numerical algebraic geometry has recently appeared in~\cite{Cummings2025Routing}.

To simplify computations, we impose the conditions $b_{ii}=1$ for each $i\in [n]$. This assumption is without loss of generality since we always have $b_{ii}>0$. Moreover, replacing $\bB$ by $\bB \mathbf{D}$ for some diagonal matrix $\mathbf{D}$ with positive entries rescales the equilibrium $\mathbf{x}_\star$ to $\mathbf{D}^{-1}\mathbf{x}_\star$ maintaining its positivity, and transforms the corresponding Jacobian $\JacAB$ in~\eqref{eq:Jac} to $\mathbf{D}^{-1}\JacAB\mathbf{D}$, which preserves its eigenvalues.

We start with the case $n=2$. Without loss of generality, we set $b_{11} = b_{22} = 1$. Thus, there remain four free parameters, viz., $a_1, a_2, b_{12}, b_{21}$. The feasibility and stability conditions are then expressed by the positivity of the polynomials $a_1 - a_2 b_{12}$, $a_2 - a_1 b_{21}$, $a_1 + a_2 - a_1 b_{21} - a_2 b_{12}$, $1 - b_{12} b_{21}$. Using \texttt{HypersurfaceRegions.jl}, we compute the regions of the hypersurface arrangement defined by these four polynomials along with the coordinate hyperplanes given by $a_1 = 0$, $a_2 = 0$, $b_{12} = 0$, and $b_{21} = 0$. There are 72 regions in total, each contractible. 
Table~\ref{table:2by 2} summarizes the region counts by sign pattern and their feasibility–stability classification (F, S). 

Among these 72 regions, exactly 8 regions satisfy both feasibility and stability conditions. The corresponding sign patterns for $(a_1, a_2, b_{12}, b_{21})$ are
\[
\begin{aligned}
&[+, +, +, +], [-, +, -, -], [+, +, -, -], [+, +, +, -], \\
&[+, -, -, -], [-, +, -, +], [+, -, +, -], [+, +, -, +].
\end{aligned}
\]
Accounting for symmetry (i.e., permuting indices 1 and 2), these patterns yield five distinct equivalence classes of feasible and stable sign patterns, matching the result in~\cite{Pierre}. Additionally, we identify 2 regions with the non-trivially impossible sign pattern $[-,-,-,-]$ corresponding to obligate mutualism (Example~\ref{ex:n2}), and 32 regions whose sign patterns match one of the above feasible/stable patterns but fail to satisfy feasibility and/or stability conditions. This shows that, as expected, the sign pattern alone does not determine feasibility and stability. Note that while there are $10$ regions with sign pattern $[+,+,+,+]$, only one of them satisfies both conditions, illustrating the complexity of confirming feasibility and stability without full polynomial analysis.

\begin{table}
\centering
\begin{tabular}{ccccc}
\toprule
   sign of $(a_1,a_2,b_{12},b_{21})$  & F, S & F, not S & S, not F & not F, not S\\
\midrule
$[+,+,+,+]$ & 1 & 0 & 2 & 7\\
$[+,+,-,-]$ & 1 & 1 & 0 & 0\\
$[+,+,+,-]$ or $[+,+,-,+]$ & 1 & 0 & 0 & 2\\
$[+,-,+,-]$ or $[-,+,-,+]$ & 1 & 0 & 0 & 2\\
$[+,-,-,-]$ or $[-,+,-,-]$ & 1 & 1 & 1 & 5\\
$[-,-,-,-]$ & 0 & 0 & 0& 2\\
$[+,-,+,+]$ or $[-,+,+,+]$ & 0 & 0 & 1 & 3\\
$[+,-,-,+]$ or $[-,+,+,-]$ & 0 & 0 & 0 & 3\\
$[-,-,+,+]$ & 0 & 1 & 1 &8\\
$[-,-,+,-]$ or $[-,-,-,+]$ & 0 & 1 & 1 & 1\\
\bottomrule
\end{tabular}
\caption{Number of regions classified by feasibility (F) and stability (S) for $n=2$, classified using \texttt{HypersurfaceRegions.jl}.}
\label{table:2by 2}
\end{table}

Next is the case $n=3$. Without loss of generality, we set $b_{11}=b_{22}=b_{33}=1$, leaving nine free variables: $a_1,a_2,a_3$ and $b_{12},b_{13},b_{21},b_{23},b_{31},b_{32}$.
From the feasibility and stability conditions, we obtain six polynomial constraints of degrees $3,3,3,3,8,3$. In principle, one could compute the regions defined by the hypersurface arrangement formed by these six polynomials together with the nine coordinate hyperplanes defined by the free variables. However, the computational complexity of this task is prohibitive: the number of critical points in these arrangements typically scales as $O(d^n)$, where $d$ is the total degree of the hypersurfaces and $n$ is the number of variables. In our setting, this leads to an estimated $O(10^{12})$ critical points, beyond current computational limits. To circumvent this difficulty, we instead fix the equilibrium values $\mathbf{x}_\star$ to be $(1,1,1)$. Under this assumption, feasibility is automatically satisfied, and our free variables reduce to $b_{12},b_{13},b_{21},b_{23},b_{31},b_{32}$, with $\ba=\bB\mathbf{x}_\star$. We then only need to consider the three polynomial constraints arising from stability conditions. Using the package \texttt{HypersurfaceRegions.jl}, we compute these regions in approximately three minutes, obtaining 1207 distinct regions. Among these, 343 regions correspond to feasible and stable solutions. In~\cite{Pierre}, the authors report six \emph{irreducible ecologies} (for which any two-species subecology is impossible or trivial), for which feasible and stable steady states are found very low probability under random sampling. In our computations, we realize all six irreducible ecologies, which together account for 30 out of the 343 feasible stable regions. This highlights the benefit of our deterministic region-computation approach over traditional random sampling methods. We provide explicit sample points corresponding to these irreducible ecologies in Table~\ref{table: sample points three by three}.

\begin{table}[h!]
\centering
\begin{tabular}{cc}
\toprule
Sign pattern of $({\bf a}, \bB)$ & Sample point \\
\hline
$[+,–,–,+,–,+,–,+,–]$ & \begin{tabular}{@{}c@{}}$[0.11,\ -0.056,\ -1.966,\ 2.949,$\\$ -3.84,\ 0.897,\ -1.954,\ 3.996,\ -6.962]$\end{tabular} \\
\hline
$[–,–,–,+,–,+,–,–,+]$ & \begin{tabular}{@{}c@{}}$[-0.361,\ -2.199,\ -0.453,\ 0.369,$\\$ -1.731,\ 0.934,\ -4.133,\ -2.816,\ 1.363]$\end{tabular} \\
\hline
$[–,–,–,+,–,–,–,–,+]$ & \begin{tabular}{@{}c@{}}$[-0.306,\ -3.197,\ -0.365,\ 0.237,$\\$ -1.543,\ -1.196,\ -3.001,\ -2.28,\ 0.915]$\end{tabular} \\
\hline
$[–,–,–,+,–,+,–,–,–]$ & \begin{tabular}{@{}c@{}}$[-0.126,\ -1.451,\ -0.222,\ 0.464,$\\$-1.59,\ 0.491,\ -2.942,\ -1.095,\ -0.127]$\end{tabular} \\
\hline
$[–,–,–,+,–,–,+,–,–]$ & \begin{tabular}{@{}c@{}}$[-0.199,\ -2.153,\ -0.204,\ 0.771,$\\$ -1.97,\ -3.547,\ 0.394,\ -1.081,\ -0.123]$\end{tabular} \\
\hline
$[–,–,–,+,–,–,–,–,–]$ & \begin{tabular}{@{}c@{}}$[-0.193,\ -3.158,\ -0.394,\ 0.57,$\\$ -1.763,\ -1.498,\ -2.66,\ -1.228,\ -0.166]$\end{tabular} \\
\bottomrule
\end{tabular}
\caption{Sample points for irreducible ecologies~\cite{Pierre} with $n=3$ obtained via region computations using \texttt{HypersurfaceRegions.jl}.}
\label{table: sample points three by three}
\end{table}

Finally we turn to the case $n=4$. As in the previous cases, we set the diagonal entries $b_{11}=b_{22}=b_{33}=b_{44}=1$, leaving 16 free variables: the four entries of $\ba$ and the 12 off-diagonal entries of the matrix $\bB$. Our hypersurface arrangement consists of eight polynomial constraints arising from the feasibility and stability conditions, together with 16 hyperplanes corresponding to the variables. To reduce complexity, we first fix the equilibrium $\xstar$ to a positive vector, as in the case $n = 3$. However, even after this reduction, the remaining $12$ free parameters make a complete region computation infeasible. We therefore introduce an additional simplification: we fix the signs of several off-diagonal entries of $\bB$ in advance. Rather than computing all regions of the arrangement, we restrict attention to verifying the realizability of specific sign patterns. The computation performed by \texttt{HypersurfaceRegions.jl} proceeds in two stages: First, it computes critical points of a suitable rational function chosen so that each region of the arrangement contains at least one critical point. Second, it tracks paths between these critical points to determine region membership. For the task of verifying the possibility of a particular sign pattern, we can restrict the computation to the first step.
We thus randomly select a positive $\xstar$, randomly choose and fix four off-diagonal entries of $\bB$ with predetermined signs, and then numerically compute critical points in one minute using numerical algebraic geometry. 
Repeating this procedure 50 times, we can successfully identify one ecology corresponding to the ordered sign pattern 
\[
[+,+,-,-,\,+,+,+,+,\,+,+,+,+,\,-,+,+,-].
\]
reported in the first plot of \cite[Supplementary Fig.~2]{Pierre}. This ecology occurs with very low probability under random sampling. A representative sample point for this irreducible ecology is given by
\[
\begin{aligned}
&[14.834,\,4.38,\,-0.033,\,-0.019,\,0.847,\,11.543,\,1.444,\,0.601,\\
&\phantom{[}0.117,\,2.662,\,0.003,\,0.589,\,-1.625,\,0.064,\,0.005,\,-1.088].
\end{aligned}
\] \\

\noindent\textbf{Acknowledgements.}
We thank Heather Harrington for guidance on research related to the Lotka--Volterra systems, Vincenzo Galgano for discussions on Grassmannians,  Oskar Henriksson for insights into the theory of chemical reaction networks, Guilherme Almeida for discussions on connections with integrable systems, and Irem Portakal for discussions on perspectives from game theory. We are grateful to Daniela Egas Santander for bringing to our attention the Lotka--Volterra type rate-firing models \cite{CurtoMorrison2023}.

\bigskip
\noindent {\bf Authors' addresses:}\\
\noindent T\"urk\"u \"Ozl\"um \c{C}elik, Max Planck Institute of Molecular Cell Biology and Genetics \\ (MPI-CBG) \& Center for Systems Biology Dresden (CSBD)
\hfill  {\tt celik@mpi-cbg.de}\\ 
\noindent Pierre A. Haas, Max Planck Institute for the Physics of Complex Systems (MPI-PKS) \& MPI-CBG \& CSBD
\hfill  {\tt haas@pks.mpg.de}\\
\noindent Georgy Scholten, MPI-CBG \& CSBD
\hfill  {\tt scholten@mpi-cbg.de}\\
\noindent Kexin Wang, Harvard University
\hfill  {\tt kexin\_wang@g.harvard.edu}\\
\noindent Gulio Zucal, MPI-CBG \& MPI-PKS \& CSBD
\hfill  {\tt zucal@mpi-cbg.de}

\end{document}